\documentclass[a4paper,12pt]{article}     
\usepackage{amsmath,amscd,amssymb}        
\usepackage{latexsym}                     
\usepackage[english, german]{babel}                

\usepackage{theorem}                 

\newtheorem{theorem}{Theorem}

\newtheorem{proposition}{Proposition}
\newtheorem{lemma}{Lemma}

\theorembodyfont{\rmfamily}

\newenvironment{definition}
{\smallskip\noindent{\bf Definition\/}:}{\smallskip\par}

\newenvironment{remark}
{\smallskip\noindent{\bf Remark\/}.}{\smallskip\par}

\newenvironment{proof}{\begin{ProofwCaption}{Proof}}{\end{ProofwCaption}}
\newenvironment{proof*}[1]{\begin{ProofwCaption}{{#1}}}{\end{ProofwCaption}}
\newenvironment{ProofwCaption}[1]%
 {\addvspace\theorempreskipamount \noindent{\it #1.}\rm}%
 {\qed \par \addvspace\theorempostskipamount}
\newcommand{\qedsymbol}{\mbox{$\Box$}}
\newcommand{\qed}{\hfill\qedsymbol}

\newcommand{\CC}{{\mathbb C}}

\newcommand{\ZZ}{{\mathbb Z}}

\newcommand{\eps}{\varepsilon}

\newcommand{\ind}{{\rm ind}\,}
\newcommand{\indGSV}{{\rm ind}_{\rm GSV}}
\newcommand{\indrad}{{\rm ind}_{\rm rad}}
\newcommand{\Conjsub}{{\rm Conjsub}\,}

\newcommand{\Sing}{{\rm Sing}\,}
\newcommand{\Eu}{{\rm Eu}\,}

\title{An equivariant version of the Euler obstruction}
\author{Wolfgang Ebeling and Sabir M.~Gusein-Zade
\thanks{Partially supported by DFG (Mercator fellowship, Eb 102/8-1), RFBR--13-01-00755
and NSh--5138.2014.1.
AMS 2010 Math. Subject Classification: 32S05, 58E40, 19A22, 58A10.
}
}
\date{}

\begin{document}
\selectlanguage{english}

\maketitle

\begin{abstract}
For a complex analytic variety with an action of a finite group and for an invariant 1-form on it, we give an equivariant version (with values in the Burnside ring of the group) of the local Euler obstruction
of the 1-form and describe its relation with the equivariant radial index defined earlier.
This leads to equivariant versions of the local Euler obstruction of a complex analytic space and of the global Euler obstruction. 
\end{abstract}

\section*{Introduction}
The classical notion of the index of an isolated singular point (zero) of a vector field
on a manifold has a number of generalizations to vector fields and 1-forms on singular varieties, e.g.,
the GSV index of a vector field on an isolated hypersurface
or complete intersection singularity, the Euler obstruction of a 1-form,
the homological index of a vector field or of a 1-form, the radial index of a vector
field or of a 1-form, etc. Surveys of these results can be found in \cite{BSS, Sur}.
The sum of the radial indices of singular points of a vector field or of a 1-form (with only isolated singularities)
on a compact variety is equal to the Euler characteristic of the variety. On a smooth manifold the Euler
obstruction and the radial index of an isolated singular point of a vector field or of a 1-form
coincide with the usual index. Thus the radial index and the Euler obstruction are related to the
Euler characteristic.

There are several equivariant versions of the Euler characteristic for spaces with actions of
a finite group $G$, in particular, the Euler characteristic of the quotient space, the orbifold Euler
characteristic offered by physicists \cite{Vafa} and its higher order generalizations \cite{AS}.
One can also define an equivariant Euler characteristic as an element of the representation ring
of the group $G$. However, a more general concept is the
equivariant Euler characteristic defined as an element of the Burnside ring $B(G)$ of the group $G$
\cite{Verdier, TtD}. The other versions of the Euler characteristic mentioned above are
specializations of this one. This leads to the problem of definition of indices of singular points
of vector fields and of 1-forms on $G$-varieties as elements of the Burnside ring $B(G)$.
For vector fields on (smooth) $G$-manifolds this was done in \cite{Luck}. For the radial index and
for the GSV one this was done in \cite{Equiv_ind}.

Here we give an equivariant version (with values in the Burnside ring $B(G)$) of the local Euler obstruction of a 1-form and describe its relation with the equivariant radial index.
This leads to equivariant versions of the local and global Euler obstructions of a complex analytic variety.
The equivariant version of the local (at a point $x$) Euler obstruction lies in the
Burnside ring of the isotropy group of the point $x$; the equivariant version of the global
Euler obstruction lives in the Burnside ring of the group under consideration.
 
\section{Euler obstruction of a 1-form}\label{sect1}
The Euler obstruction of a (real) 1-form $\omega$ on a germ $(V,x)$ of a complex analytic variety was essentially defined in \cite{MacPh}.

Let $V$ be a complex analytic variety endowed with a Whitney stratification $\{ V_i \}_{i \in I}$ and let $\omega$ be a complex (continuous) 1-form on it (i.e.\ the restriction of a 1-form on an ambient smooth space). A point $x \in V$ is a singular point of the 1-form $\omega$ on $V$ if it is a singular point (zero) of the restriction of the 1-form $\omega$ to the stratum $V_i$ containing $x$. In particular, all zero dimensional strata of the stratification are regarded as singular points of $\omega$.

Assume that the form $\omega$ on $(V,x)$ has an isolated singular point at $x$. Let $B_\eps(x)$ be a ball of a small radius $\eps >0$ around $x$ such that representatives of $(V,x)$ and $\omega$ are defined in $B_\eps(x)$ and $\omega$ has no singular points on $V \setminus \{ x \}$ inside this ball. Let $\nu : \widehat{V} \to V$ be the Nash blow up of the variety $V$ and let $\widehat{T}$ be the Nash bundle over $\widehat{V}$: see, e.g., \cite{MacPh, GD}. The 1-form $\omega$ defines a section $\widehat{\omega}$ of the (real) dual Nash bundle $\widehat{T}^*$ which has no zeros outside of $\nu^{-1}(x)$. The {\em local Euler obstruction} ${\rm Eu}(\omega,V,x)$ of the 1-form $\omega$ on $(V,x)$ is the obstruction to extend the non-zero section $\widehat{\omega}$ from the preimage of the sphere $S_\eps=\partial B_\eps(x)$ to the preimage of the ball $B_\eps(x)$.

The {\em local Euler obstruction} ${\rm Eu}(V,x)$ {\em of the germ} $(V,x)$ of a complex analytic variety is the local Euler obstruction of the radial 1-form $d|r|^2$ on it, where $r$ is the distance to the point $x$. If $\{ V_i \}_{i \in I}$ is a Whitney stratification of a complex analytic variety $V$ and $V_i \subset \overline{V}_j$ then the local Euler obstruction ${\rm Eu}(\overline{V}_j,p)$ at any point $p \in V_i$ does not depend on $p$ (see \cite[Corollaire~10.2]{BS}, \cite[Theorem~8.1.1]{BSS}). We shall denote it by ${\rm Eu}(V_j,V_i)$. It is equal to the local Euler obstruction ${\rm Eu}(N_{ij},p)$ of a normal slice $N_{ij}$ of the variety $\overline{V}_j$ to the stratum $V_i$ (at the point $p$) (see \cite[Section~3]{BLS}). If $V_i \not\subset \overline{V}_j$, we assume ${\rm Eu}(V_j,V_i)$ to be equal to zero.

To be consistent with \cite{GD}, we prefer to consider indices of complex valued 1-forms on a variety.
The difference in the definition emerges from the difference in the orientation of the (complex) dual
space to $\CC^n$ and the orientation of the (real) dual space to $\CC^n$ with its complex structure forgotten
(see, e.g., \cite{Proportionality} or \cite[Example on page 238]{GD}). This leads to the following definitions.

\begin{definition}
 The Euler obstruction and the radial index of a complex valued 1-form $\omega$ on a germ of a purely 
 $n$-dimensional complex analytic variety $(V,x)$ are defined by
  \begin{eqnarray*}
  {\rm Eu}(\omega;V,x)&=&(-1)^n {\rm Eu}({\rm Re}\, \omega;V,x),\\
  \indrad(\omega;V,x)&=&(-1)^n \indrad({\rm Re}\,\omega;V,x),
  \end{eqnarray*}
  where ${\rm Re}\,\omega$ is the real part of the 1-form $\omega$.
\end{definition}

The same (sign) convention should be applied to the definition of the equivariant radial index
in \cite{Equiv_ind}. (There they were defined for real-valued 1-forms.)

\section{Burnside ring}\label{sect2}

Let $G$ be a finite group. Let $\Conjsub G$ be the set of conjugacy classes of subgroups of $G$.
A $G$-set is a set with an action of the group $G$.
A $G$-set is {\em irreducible} if the action of $G$ on it is transitive.
Isomorphism classes of irreducible $G$-sets are in one-to-one correspondence with
elements of $\Conjsub G$: to the conjugacy class $[H]$ containing
a subgroup $H\subset G$ one associates the isomorphism class $[G/H]$
of the $G$-set $G/H$. The {\em Burnside ring} $B(G)$ of $G$ is the
Grothendieck ring of finite $G$-sets, i.e.\
the abelian group generated by the isomorphism classes of finite $G$-sets
modulo the relation $[A\amalg B]=[A]+[B]$ for finite $G$-sets $A$ and $B$. The multiplication
in $B(G)$ is defined by the 
cartesian product. As an abelian group, $B(G)$
is freely generated by the isomorphism classes of irreducible $G$-sets,
i.e.\ each element of $B(G)$ can be written in a unique way as 
$\sum\limits_{[H]\in \Conjsub G}a_{[H]}[G/H]$ with $a_{[H]}\in\ZZ$.
The element $1$ in
the ring $B(G)$ is represented by the $G$-set $[G/G]$ consisting of one point (with the trivial $G$-action).

There is a natural homomorphism from the Burnside ring $B(G)$ to the ring $R(G)$
of representations of the group $G$ which sends a $G$-set $X$ to the (vector)
space of functions on $X$. If $G$ is cyclic, then this homomorphism is injective.
In general, it is neither injective nor surjective.

For a subgroup $H\subset G$ one has two natural homomorphisms of abelian groups:
the {\em restriction map} $\mbox{R}_{H}^{G}: B(G) \to B(H)$, which sends a $G$-set X to the same set considered
with the action of $H$, and the {\em induction map} $\mbox{I}_{H}^{G}: B(H) \to B(G)$, which sends an $H$-set $X$
to the $G$-set $G\times X/(gh, x)\sim (g, hx)$. The first one is a ring homomorphism, the latter one is not.
The reduction homomorphism $\mbox{R}_{\{e\}}^{G}$ sends a virtual $G$-set $A$ to the number $\vert A\vert$
of elements of $A$.

The main requirement for an equivariant version with values in the Burnside ring $B(G)$
of an integer valued invariant is that it reduces to the usual (non-equivariant) one
under the reduction homomorphism $\mbox{R}_{\{e\}}^{G}$.

For a ``sufficiently good'' $G$-space $V$, say, a subanalytic variety, one has a natural
equivariant version of the Euler characteristic with values in the Burnside ring $B(G)$ of the group $G$:
\cite{Verdier, TtD}. It can be defined by the equation
$$
\chi^G(V)=\sum\limits_{[H]\in{\rm Conjsub\,}G}\chi(V^{([H])}/G)[G/H]\,,
$$
where $V^{([H])}$ is the set of points $x\in V$ with the isotropy subgroup $G_x=\{g\in G: gx=x\}$
such that  $[G_x]=[H]$. The reduction homomorphism $R^G_{\{e\}}:B(G)\to B(\{e\})\cong \ZZ$
sends the equivariant Euler characteristic $\chi^G(V) \in B(G)$ to the usual Euler characteristic $\chi(V) \in \ZZ$.

The equivariant radial index of a vector field or of a $1$-form defined in \cite{Equiv_ind} takes values
in the Burnside ring $B(G_x)$ of the isotropy subgroup $G_x$ of an isolated singular point $x$. For a compact subanalytic variety $V$ and for a $G$-invariant 1-form
$\omega$ on it, one has
$$
\sum\limits_{\overline{p}\in({\rm Sing\,}\omega)/G}I_{G_p}^G\left(\indrad^{G_p}(\omega;V,p)\right)=
\chi^G(V)\in B(G)\,,
$$
where $p$ is a point of the $G$-orbit $\overline{p}$.

\section{An equivariant version of the Euler obstruction}\label{sect3}
Let $(V,x)\subset(\CC^N,x)$ be a germ of a complex analytic variety with an action of a finite group $H$
and let $\{V_i\}_{i \in I}$  be an $H$-invariant Whitney stratification of it.
(It is assumed that all the points of a stratum $V_i$ have one and the same conjugacy class $[H_i]$ of the isotropy subgroups and
that $V_i/H$ is connected for each $i$. We assume that $V=\overline{V}_{i_0}$ for some $i_0 \in I$.) The strata $V_i$
are partially ordered: $V_i \prec V_j$ (we shall write $i \prec j$) iff $V_i \subset \overline{V_j}$
and $V_i \neq V_j$; $i \preceq j$ iff $i \prec j$ or $i=j$. 

Let $\omega$ be a germ of an $H$-invariant complex
1-form on $(V,x)$ with an isolated singular point at $x$. Let $B_\eps(x)$ be a ball of a small radius
$\eps$ around $x$ such that representatives of $V$ and of $\omega$ are defined in $B_\eps(x)$
and the 1-form $\omega$ has no singular points on $V\setminus\{x\}$ inside $B_\eps(x)$;
let $S_{\eps}:=\partial B_\eps(x)$.
There exists an $H$-invariant 1-form $\widetilde\omega$ on $V\cap B_\eps(x)$ such that
\begin{enumerate}
 \item[1)] the 1-form $\widetilde\omega$ coincides with $\omega$ on a neighbourhood of $V\cap S_\eps$;
 \item[2)] the 1-form $\widetilde\omega$ has finitely many singular points on $V\cap B_\eps(x)$;
 \item[3)] in a neighbourhood of each singular point $p\in V\cap B_\eps(x)$, $p\in V_i=:V_{(p)}$,
 the 1-form $\widetilde\omega$ is a radial extension of its restriction to the stratum $V_{(p)}$;
\end{enumerate}
see \cite{Equiv_ind}.

\begin{definition}
 The $H$-{\em equivariant local Euler obstruction} of the 1-form $\omega$ on $(V,x)$ is
 \begin{equation}
  \Eu^H(\omega;V,x)=\sum_{\overline{p}\in ({\Sing}\widetilde{\omega})/H}
  (-1)^{\dim{V}-\dim{V(p)}}\Eu(V,V_{(p)})\cdot\ind(\widetilde{\omega}_{\vert V(p)};V_{(p)},p)[Hp]\,,
 \end{equation}
 where $p$ is a point of the orbit $\overline{p}=Hp$, $\ind(\cdot)$ is the usual index of a $1$-form
 on a smooth manifold.
\end{definition}

The fact that the $H$-equivariant local Euler obstruction $\Eu^H(\omega;V,x)$ is well defined
follows, e.g., from \cite[Proposition 1]{Equiv_ind}. Since the (usual) local Euler obstruction
of a 1-form satisfies the law of conservation of number (with respect to a deformation of the 1-form),
one has
$$
\Eu(\omega;V,x)=\sum_{p\in {\rm Sing\,}\widetilde{\omega}}
\Eu(\widetilde{\omega};V,p)\,.
$$
Since, in a neighbourhood of a singular point $p$, the 1-form $\widetilde{\omega}$
is a radial extension of its restriction to the corresponding stratum $V_{(p)}$, one has
$\Eu(\widetilde{\omega};V,p)=\Eu(V,V_{(p)})\cdot\ind(\widetilde{\omega}_{\vert V(p)};V_{(p)},p)$
(see, e.g., \cite{Proportionality}). Therefore
$R^H_{\{e\}}\Eu^H(\omega;V,x)=\Eu(\omega;V,x)$. (This was the main requirement for the definition
of the $H$-equivariant local Euler obstruction with values in the Burnside ring $B(H)$.)

\begin{remark}
 The $H$-equivariant radial index  $\indrad^H(\omega;V,x)$ of the 1-form $\omega$,
 defined (for the real setting) in \cite{Equiv_ind}, is equal to
$$
\sum_{\overline{p}\in ({\Sing}\widetilde{\omega})/H}
(-1)^{\dim{V}-\dim{V(p)}}\ind(\widetilde{\omega}_{\vert V(p)};V_{(p)},p)[Hp]\,.
$$
(Pay attention to the sign the origin of which was described earlier.) Thus the equivariant radial index
and the equivariant local Euler obstruction should be related. 
\end{remark}

Let $\zeta_{ij}$ be the zeta function of the partially ordered set ({\em poset}) of the strata $V_i$:
$$
\zeta_{ij}=
\begin{cases}
1 &{\rm if\ } i\preceq j,\\
0 &{\rm otherwise.}
\end{cases}
$$
and let $\mu_{ij}$ be the M\"obius function of this poset, i.e.\ $\mu_{ij}=0$ for $i\not\preceq j$ and
$\sum_j \zeta_{ij}\mu_{jk}=\delta_{ik}$ (the Kronecker delta): see \cite{Hall}.

\begin{theorem}\label{theo}
For an $H$-invariant complex 1-form $\omega$ on the germ $(V,x)$ one has
  \begin{equation}\label{eqn-theo}
  \Eu^H(\omega;V,x)=\sum_{j}
\left(
\sum_{i}\mu_{ji} \Eu(V, V_i)
\right)
\indrad^H(\omega_{\vert\overline{V}_j};\overline{V}_j,x)\,.
 \end{equation}
\end{theorem}

\begin{proof}
Let
$$
s_i:=(-1)^{\dim{V}-\dim{V_i}}\sum_{\overline{p}\in ({\Sing}\widetilde{\omega}\cap V_i)/H}
\ind(\widetilde{\omega}_{\vert V_i};V_i,p)\,.
$$
As it was mentioned above, $s_i$ does not depend on the choice of the 1-form $\widetilde{\omega}$
possessing the properties 1)-3). One has
\begin{equation}\label{indrad}
\indrad^H(\omega_{\vert\overline{V}_k};\overline{V}_k,x)=\sum_{i} s_i[H/H_i]\zeta_{ik}\,,
\end{equation}
\begin{equation}\label{Euler}
\Eu^H(\omega_{\vert\overline{V}_k};\overline{V}_k,x)=\sum_{i\preceq k} \Eu(V_k,V_i)s_i[H/H_i]\,.
\end{equation}
From (\ref{indrad}) it follows that
$$
s_i[H/H_i]=\sum_j \mu_{ji} \cdot \indrad^H(\omega_{\vert\overline{V}_j};\overline{V}_j,x)\,.
$$
From (\ref{Euler}) one gets
\begin{equation}\label{eqn-strat}
\Eu^H(\omega_{\vert \overline{V}_k};\overline{V}_k,x)=\sum_{j}
\left(
\sum_{i}\mu_{ji} \Eu(V_k, V_i)
\right)
\indrad^H(\omega_{\vert\overline{V}_j};\overline{V}_j,x)\,.
\end{equation}
Applying (\ref{eqn-strat}) to $k=i_0$ ($\overline{V}_{i_0}=V$) one gets Equation (\ref{eqn-theo}).
\end{proof}

Equation (\ref{eqn-theo}) means that the equivariant local Euler obstruction of an $H$-invariant 1-form $\omega$ on $(V,x)$ is a linear
combination of the equivariant radial indices of $\omega$ on the closures $\overline{V}_j$ of the strata
with integer coefficients depending only on the stratification. A similar equation (for $H=\{e\}$) was
obtained in \cite[Corollary 1]{GD} in other terms. 
Let $N_{ij}$ be the normal slice of the variety $\overline{V}_j$ to the stratum $V_i$ (at a point $p$ of it).
(We assume that for $i \preceq j\preceq k$ one has $N_{ij}=N_{ik}\cap \overline{V}_j$.)
Let $n_{ij}$ be defined by
$$
n_{ij}:=\indrad(d\ell; N_{ij}, p)
$$
for the differential $d\ell$ of a generic linear function $\ell: \CC^N \to \CC$, and let $m_{ij}$ be the M\"obius
inverse of the function $n_{ij}$ on the poset of the strata $V_i$ (see \cite{Hall}). In \cite[Corollary 1]{GD} the
following equation was obtained:
\begin{equation}\label{old}
 {\rm Eu}(\omega; V,0)=\sum_j m_{ji_0}\cdot\indrad(\omega; \overline{V}_j,0)\,.
\end{equation}
The invariants used there and the proof cannot be
adapted to the equivariant setting: the deformation of the 1-form considered there will not, in general,
be $H$-invariant. The following statement establishes the relation between Equations (\ref{eqn-theo}) and (\ref{old}).

\begin{proposition}
 One has
 \begin{equation}\label{eqn-prop}
 \sum_i \mu_{ji}{\rm Eu}(V_k,V_i)=m_{jk}.  
 \end{equation}
\end{proposition}

\begin{proof}
 Let us denote the left hand side of (\ref{eqn-prop}) by $\widetilde{m}_{jk}$ and let $\widetilde{n}_{ij}$
 be the M\"obius inverse of the function $\widetilde{m}_{jk}$ on the poset of strata. For $H=\{e\}$,
 Theorem~\ref{theo} applied to the normal slice $N_{ik}$ of $\overline{V}_k$ to the stratum $V_i$
 (at a point $p \in V_i$) gives
 $$
 {\rm Eu}(\omega_{\vert N_{ik}};N_{ik},p)=
 \sum_j \widetilde{m}_{jk} \cdot \indrad(\omega_{\vert N_{ij}};N_{ij},p)\,.
 $$
 This implies that
 $$
 \indrad(\omega_{\vert N_{ik}};N_{ik},p)=
 \sum_j \widetilde{n}_{jk} \cdot  {\rm Eu}(\omega_{\vert N_{ij}};N_{ij},p)\,.
 $$
 For a generic linear function $\ell$ one has
 $$
 {\rm Eu}(d\ell_{\vert N_{ij}};N_{ij},p)=
 \begin{cases}
  0 & {\rm if \ } i\prec j,\\
  1 & {\rm if \ } i=j.
 \end{cases}
 $$
 Therefore
 $$
 \widetilde{n}_{jk}=\indrad(d\ell_{\vert N_{jk}};N_{jk},p)=n_{jk}\,.
 $$
\end{proof}

One can define the equivariant local Euler obstruction ${\rm Eu}^H(\eta;V,x)$ of a real $1$-form $\eta$
on $(V,x)$. As it was explained in Section~\ref{sect1}, it should be defined as
$(-1)^{\dim V}{\rm Eu}^H(\eta^{\CC};V,x)$ for the complexification $\eta^{\CC}$ of the 1-form $\eta$.
In particular, the equivariant local Euler obstruction ${\rm Eu}^H(V,x)$ of a complex analytic space
$V$ at a point $x$ is the equivariant local Euler obstruction ${\rm Eu}^H(d\vert r\vert^2;V,x)$
of the (real) radial $1$-form $d\vert r\vert^2$.

In this way the notion of the global Euler obstruction of a quasi-projective subvariety $V$ in $\CC^N$
defined (in the non-equivariant setting), e.g., in \cite{STV} has a natural equivariant generalization
for a $G$-invariant subvariety $V$ as an element of the Burnside ring $B(G)$:
$$
{\rm Eu}^G(V)=\sum_{\overline{p}\in ({\rm Sing\,}\eta_{\rm rad})/G} I^G_{G_p} ( {\rm Eu}^{G_p}(\eta_{\rm rad};V,p))
\in B(G)
$$
for a real $1$-form $\eta_{\rm rad}$ on $\CC^N$ which is radial at infinity.
In the obvious way one gives the definition of the global Euler obstruction for compact analytic varieties
and gets an equivariant version of it as an element of the Burnside ring. The definition of the
equivariant global Euler obstruction (and of the local one as well) cannot be obtained by a slight
modification of the usual one because of the lack of a notion of the topological obstruction 
as an element of the Burnside ring in an equivariant setting.

Let $V$ be either a closed quasi-projective variety in $\CC^N$ or a compact complex analytic variety and let $\{W_j\}_{j \in J}$ be a $G$-invariant Whitney stratification of $V$.

\begin{proposition} \label{prop:globalEuler}
One has
\begin{equation}
{\rm Eu}^G(V) = \sum_j {\rm Eu}(V,W_j) \cdot \chi^G(W_j) \label{eqn-Glob}
\end{equation}
\end{proposition}

We shall use the following lemma. Let $M$ be a (real) manifold with boundary $\partial M$. A 1-form $\eta$ on $M$ will be called {\em pointing inwards} if, at each point $p \in \partial M$ such that the restriction of $\eta_p$ to $\partial M$ vanishes, the value of $\eta_p$ on a vector pointing inwards $M$ is positive.

\begin{lemma} \label{lem:Hopf}
Let $\eta$ be a 1-form on $M$ pointing inwards with only isolated singular points. Then one has
$$
\sum_{p \in {\rm Sing}\, \eta} \ind(\eta; M,p) = \chi(M \setminus \partial M).
$$
\end{lemma}

\begin{proof*}{Proof of Proposition~\ref{prop:globalEuler}}
We shall write the proof for a compact complex analytic variety. For a (closed) subvariety in $\CC^N$ one can apply the same arguments taking into account that, in the definition of the global Euler obstruction, instead of a 1-form $\eta_{\rm rad}$ which is radial at infinity one can use the 1-form $-\eta_{\rm rad}$. It is sufficient to prove (\ref{eqn-Glob}) for a subpartition of the stratification $\{W_j\}_{j \in J}$. Therefore one can assume that the isotropy subgroups of the points of a stratum $W_j$ are all conjugate to the same subgroup $H_j$. To compute the global Euler obstruction one can use a $G$-invariant 1-form $\eta$ which, in a neighbourhood of each singular point is the radial extension of its restriction to the corresponding stratum. For each stratum $W_j$ let $T_j$ be a small tubular neighbourhood of the union of adjacent strata. Its complement $\stackrel{\circ}{W}_j := W_j \setminus T_j$ can be considered as a manifold with boundary. According to Lemma~\ref{lem:Hopf} the sum of the indices of $\eta_{|\stackrel{\circ}{W}_j}$ is equal to $\chi(W_j)$. Therefore the number of orbits of singular points of $\eta$ counted with multiplicities is equal to $\chi(W_j/G)$. According to the definition of the equivariant Euler obstruction one has
\begin{eqnarray*}
{\rm Eu}^G(V) & = & \sum_j \sum_{\overline{p} \in ({\rm Sing}\, \eta \cap W_j)/G} {\rm Eu}(V,W_j) \cdot \ind(\eta; W_j,p)[G/H_j] \\
& = & \sum_j {\rm Eu}(V,W_j) \cdot \chi(W_j/G)[G/H_j] \\
 & = & \sum_j {\rm Eu}(V,W_j) \chi^G(W_j).
\end{eqnarray*}
\end{proof*}

Let $(V,x)\subset(\CC^N,x)$ be an $H$-invariant $n$-dimensional isolated complete intersection singularity
defined by $H$-invariant equations. In \cite{Equiv_ind}, there was defined the notion of the
equivariant GSV index $\indGSV^H(\omega;V,x)$ of an $H$-invariant $1$-form.
Both the equivariant local Euler obstruction and the equivariant GSV index satisfy the law of conservation of
number and they coincide on a smooth manifold. This implies that the difference 
${\rm Eu}^H(\omega;V,x)-\indGSV^H(\omega;V,x)$ does not depend on the $1$-form $\omega$.
Comparing this difference with the same one for the radial 1-form, one gets the following statement.

\begin{proposition}
 One has
 $$
 {\rm Eu}^H(\omega;V,x)=\indGSV^H(\omega;V,x)+(-1)^n{\rm Eu}^H(V,x)-(-1)^n \chi^H(M)\,,
 $$
 where $M$ is the Milnor fibre of the isolated complete intersection singularity $(V,x)$.
\end{proposition}


\bigskip
\noindent Leibniz Universit\"{a}t Hannover, Institut f\"{u}r Algebraische Geometrie,\\
Postfach 6009, D-30060 Hannover, Germany \\
E-mail: ebeling@math.uni-hannover.de\\

\medskip
\noindent Moscow State University, Faculty of Mechanics and Mathematics,\\
Moscow, GSP-1, 119991, Russia\\
E-mail: sabir@mccme.ru


\begin{thebibliography}{10}
\bibitem{AS} M.~Atiyah, G.~Segal:
On equivariant Euler characteristics. 
J. Geom. Phys. {\bf 6}, no.4, 671--677 (1989).

\bibitem{BLS} J.-P.~Brasselet, L\^e D.~T., J.~Seade: Euler obstruction
and indices of
vector fields. Topology {\bf 39}, 1193--1208 (2000).

\bibitem{BS} J.-P.~Brasselet, M.-H.~Schwartz: Sur les classes de Chern
d'un ensemble analytique complexe. In: Caract\'eristique d'Euler-Poincar\'e,
Ast\'erisque {\bf 82--83}, 93--147 (1981).

\bibitem{Proportionality} J.-P.~Brasselet, J.~Seade, T.~Suwa:
Proportionality of indices of 1-forms on singular varieties.
In: Singularities in geometry and topology 2004, 49--65,
Adv. Stud. Pure Math., 46, Math. Soc. Japan, Tokyo, 2007.

\bibitem{BSS} J.-P.~Brasselet, J.~Seade, T.~Suwa: Vector fields on singular varieties.
Lecture Notes in Mathematics, Vol.~1987. Springer-Verlag, Berlin, 2009.


\bibitem{Vafa} L.~Dixon, J.~Harvey, C.~Vafa, E.~Witten:
Strings on orbifolds. Nucl. Phys. {\bf B 261}, 678--686 (1985), II. Nucl. Phys. {\bf B 274}, 285--314 (1986).

\bibitem{GD} W.~Ebeling, S.M.~Gusein-Zade: Radial index and Euler obstruction of a 1-form
on a singular variety. Geom. Dedicata {\bf 113}, 231--241 (2005).

\bibitem{Sur} W.~Ebeling, S.M.~Gusein-Zade: Indices of vector fields and 1-forms on singular varieties. In: Global aspects of complex geometry, 129--169,
Springer, Berlin, 2006.


\bibitem{Equiv_ind} W.~Ebeling, S.M.~Gusein-Zade:
Equivariant indices of vector fields and 1-forms.
arXiv: 1307.2054.




\bibitem{Hall} M.~Hall Jr.: Combinatorial theory. Second edition.
Wiley-Interscience Series in Discrete Mathematics. A Wiley-Interscience Publication.
John Wiley \& Sons, Inc., New York, 1986.


\bibitem{Luck} W.~L\"uck, J.~Rosenberg: The equivariant Lefschetz fixed point theorem
for proper cocompact $G$-manifolds. In: High-dimensional manifold topology, 322--361,
World Sci. Publ., River Edge, NJ, 2003.

\bibitem{MacPh} R.~MacPherson: Chern classes for
singular varieties. Annals of Math. \textbf{100}, 423--432 (1974).



\bibitem{STV} J.~A.~Seade, M.~Tibar, A.~Verjovsky: Global Euler obstruction and polar invariants. Math. Ann. {\bf 333}, 393--403 (2005).

\bibitem{TtD} T.~tom Dieck: Transformation groups and representation theory.
Lecture Notes in Mathematics, 766, Springer, Berlin, 1979.

\bibitem{Verdier} J.-L.~Verdier: Caract\'eristique d'Euler-Poincar\'e. Bull. Soc. Math. France {\bf 101},
441--445 (1973).


\end{thebibliography}
\end{document}